\documentclass[reqno]{amsart}
\usepackage{amssymb}
\usepackage{amsthm}
\usepackage{bbm}
\usepackage{graphicx}

\usepackage[colorlinks=true]{hyperref}
\usepackage[all]{xypic}

\usepackage[final]{epsfig}
\usepackage{psfrag}
\usepackage{epstopdf}

\newtheorem{thm}{Theorem}[section]
\newtheorem{theorem}[thm]{Theorem}
\newtheorem{proposition}[thm]{Proposition}
\newtheorem{corollary}[thm]{Corollary}
\newtheorem{lemma}[thm]{Lemma}

\newtheorem{prop}[thm]{Proposition}

\theoremstyle{definition}
\newtheorem{definition}{Definition}[section]

\theoremstyle{observation}

\theoremstyle{definition}
\newtheorem{example}{Example}[section]

\newcommand{\A}{\mathcal{A}}
\newcommand{\B}{\mathcal{B}}

\newcommand{\M}{\mathcal{M}}

\newcommand{\defin}[1]{{\it #1}}
\newcommand{\R}{\mathbb{R}}
\newcommand{\N}{\mathbb{N}}
\newcommand{\Q}{\mathbb{Q}}
\newcommand{\C}{\mathbb{C}}
\newcommand{\Z}{\mathbb{Z}}
\newcommand{\D}{\mathbb{D}}

{\bf}{\it}
{\bf}{\it}
\newtheorem*{riemann-mapping-theorem}{Riemann Mapping  Theorem}{\bf}{\it}
{\bf}{\it}
{\bf}{\it}
{\bf}{\it}
{\bf}{\it}
{\bf}{\it}
{\bf}{\it}
{\bf}{\it}
{\bf}{\it}
{\bf}{\it}

\newenvironment{pf*}[1]{\proof[#1]}{\endproof}
\usepackage{euscript}

\newcommand{\cal}[1]{{\mathcal #1}}

\newcommand{\beq}{\begin{equation}}
\newcommand{\eeq}{\end{equation}}

\newcommand{\riem}{{\widehat{\CC}}}

\newcommand{\dist}{\operatorname{dist}}

\newcommand{\cl}{\operatorname{cl}}

\newcommand{\eps}{\epsilon}

\numberwithin{equation}{section}

\newcommand{\cA}{{\cal A}}

\newcommand{\cW}{{\cal W}}
\newcommand{\cM}{{\cal M}}

\renewcommand{\cR}{{\cal R}}

\newcommand{\cE}{{\cal E}}

\newcommand{\CC}{{\mathbb C}}
\newcommand{\RR}{{\mathbb R}}

\newcommand{\ZZ}{{\mathbb Z}}
\newcommand{\NN}{{\mathbb N}}
\newcommand{\DD}{{\mathbb D}}

\newcommand{\QQ}{{\mathbb Q}}

\newcommand{\ignore}[1]{{}}

\title[Computability of Mandelbrot-like set]{Non computable Mandelbrot-like set for a one-parameter complex family}
\author{Daniel Coronel, Cristobal Rojas, and Michael Yampolsky}

 \thanks{D.C. and C.R were partially supported by project DI-782-15/R Universidad Andr\'es Bello and Basal PFB-03 CMM-Universidad de Chile. M.Y. was partially supported by NSERC Discovery grant.}
 
\begin{document}
\maketitle

\begin{abstract} We show the existence of  computable complex numbers $\lambda$ for which the bifurcation locus of the  one parameter complex family $f_{b}(z) = \lambda z + b z^{2} + z^{3}$ is not Turing computable. 
\end{abstract}

% !TEX root = Mandelbrot3D.tex

\section{Introduction}
For a complex quadratic map $P_{c}(z)= z^{2} + c$,  recall that the filled Julia set $K(P_{c})$ corresponds to the set of points $z\in \CC$ whose orbit under iterations by $P_{c}$ remains bounded, and that the Julia set $J(P_{c})$ is defined as the boundary of $K(P_{c})$.  Let us also recall that the Mandelbrot set $\M$ is defined to be the \emph{connectedness locus} of the family $P_{c}(z)=z^{2}+c, \,\, c\in \CC$:
the set of complex parameters $c$ for which the Julia set $J(P_{c})$ is connected. The boundary of $\M$ corresponds to the parameters near which the geometry of the Julia set undergoes a dramatic change. For this reason, its boundary $\partial \M$ is referred to as the \emph{bifurcation locus}. 
The Mandelbrot set  is widely known for the spectacular beauty of its fractal structure, and an enormous amount of effort has been made in order to understand its topological and geometrical properties. This effort has greatly relied on computer simulations, and it is most natural to ask {\it whether these simulations can be trusted}. A form of this question was first asked by Penrose in \cite{penrose1991emperor} and has been a subject of much interest.

The central open conjecture in complex dynamics is known as Density of Hyperbolicity Conjecture. This conjecture is widely expected to be true, and postulates that $\cM$ is the closure of the open set of parameter values $c$ for which $P_c$ exhibits {\it hyperbolic dynamics}. The latter simply means that
$|DP^n_c|>1$ on a neighborhood of $J(P_c)$ for some $n\in\NN$.
In \cite{Hert}, Hertling demonstrated that Density of Hyperbolicity Conjecture
implies that Mandelbrot set $\cM$ as well as its boundary, the bifurcation locus $\partial\M$, are rigorously computable.

In this paper we show that such a computability property cannot be taken for granted. We consider a different one-parameter family of complex dynamical systems, studied by X.~Buff and C.~Henriksen in \cite{BH}:
$$
f_b = \lambda z + bz^2 + z^3, \qquad b\in \mathbb{C},
$$ 
where $\lambda = e^{2i\pi \theta} \in \mathbb{C}$ satisfies $|\lambda|=1$.  We denote by %$K(f_b)$ the filled-in Julia set of $f_b$, $J(f_b)$ its Julia set, and
$\M_\lambda$ the connectedness locus of the family, that is, the set of complex parameters $b$ for which the Julia set $J(f_{b})$ is connected. 
%\begin{align*}
%K(f_b)&=\{z\in \mathbb{C} | (f^n_b(z))_{n\in\N} \text{ is bounded } \},\\
%J(f_b)&=\partial K(f_b), \text{ and }\\
%\M_\lambda &= \{b\in \mathbb{C} | J(f_b) \text{  is connected  }\}.
%\end{align*}

Our main result is the following. 

\medskip
\noindent\textbf{Main Theorem.} \emph{There exists a computable (by an explicit algorithm) value of $\lambda$ such that the bifurcation locus $\partial \M_{\lambda}$ is not computable}. 
\bigskip

\noindent
A principal result of  \cite{BH} is that for each $\lambda$ of modulus 1, the bifurcation locus $\partial \M_\lambda$ contains quasi-conformal copies of the quadratic Julia set $J(\lambda z + z^2)$ (see Figure~1 for an illustration). The proof of the Main Theorem relies on a  computable version of this statement, which is our principal technical result.

\begin{figure}[ht]
\centerline{\includegraphics[width=10cm]{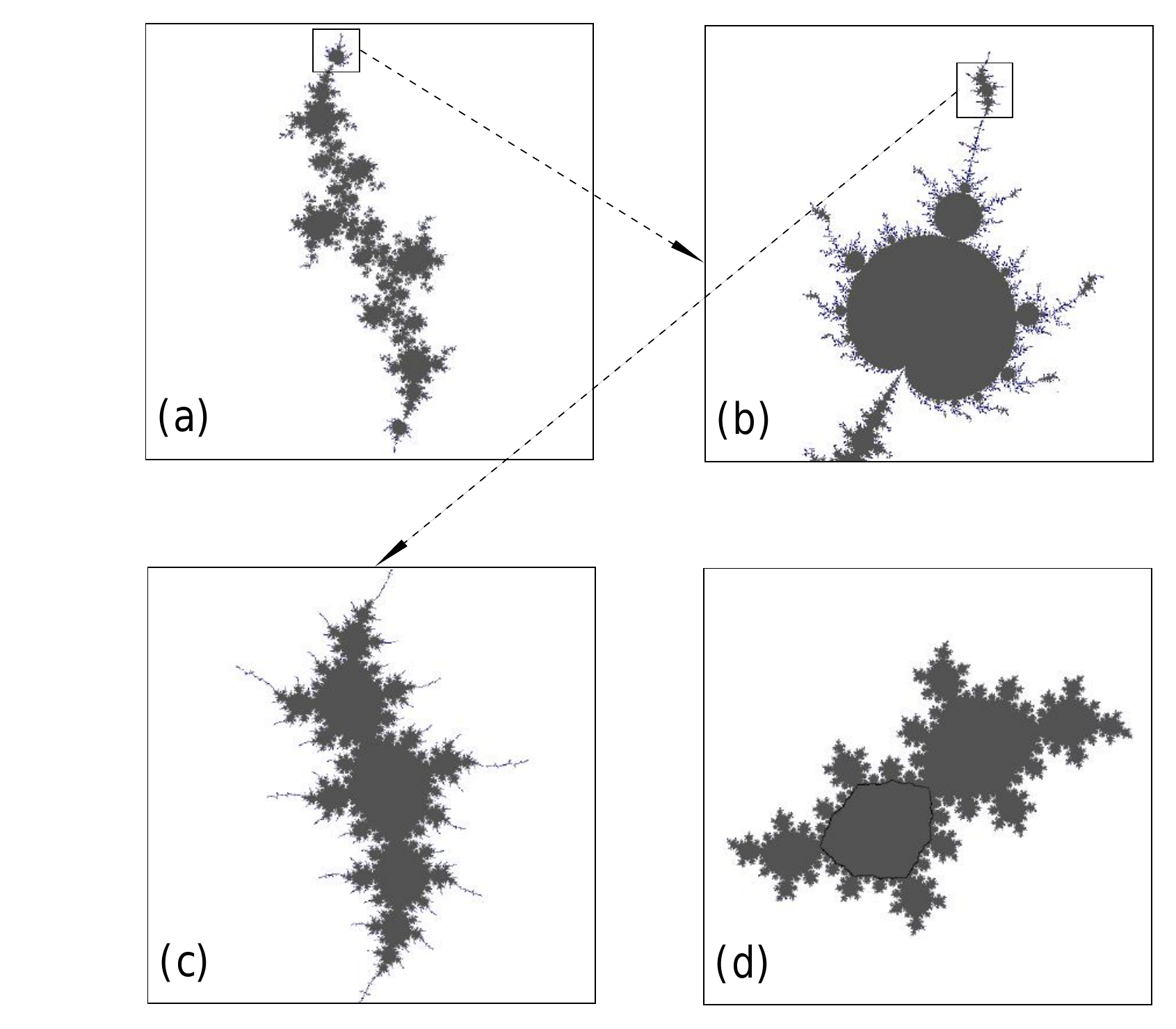}}
\caption{A reproduction of Figure~1 from \cite{BH} (computed using {\it Mandel} by W.~Jung). $\lambda=e^{2\pi i\theta}$, where $\theta=(\sqrt{1}+1)/2$ is the golden mean. (a) is the set $\M_\lambda$; (b) and (c) are its successive blow-ups, showing a copy of the Julia set $J(\lambda z+z^2)$ hidden inside. (d) is the Julia set $J(\lambda z+z^2)$; the boundary of the Siegel disk is highlighted. Note that this Julia set is actually computable  (see \cite{BY-MMJ}).
\label{figure1}}
\end{figure}

% !TEX root = Mandelbrot3D.tex

 \section{Preliminaries on Computability}
 
 \subsection{Rudiments of Computable Analysis and applications to Julia sets}
We give a very brief summary of relevant notions of Computability Theory and Computable Analysis. For a more in-depth
introduction, the reader is referred to \cite{Wei, BY-book}.
As is standard in Computer Science, we formalize the notion of
an algorithm as a {\it Turing Machine} \cite{Tur}.
%
%It is more intuitively familiar, and provably equivalent,
%to think of an algorithm as a program written in any standard programming language.
%One then readily sees that there is only a countable number of possible algorithms. Fixing the language, we can
%enumerate them all in an effective way (for instance, lexicographically). Given such an ordered list $(\cA_n)_{n=1}^\infty$ of all
%algorithms, the index $n$ is usually called the \emph{G\"odel number} of the algorithm $\cA_{n}$.

We will call a  function $f:\NN\to\NN$  \emph{computable} (or {\it recursive}), if there exists an algorithm $\cA$ which, upon input $n$, outputs $f(n)$.

Extending algorithmic notions to functions of real numbers was pioneered by Banach and Mazur \cite{BM,Maz}, and
is now known under the name of {\it Computable Analysis}. Let us begin by giving the modern definition of the notion of computable real
number,  which goes back to the seminal paper of Turing \cite{Tur}. By identifying $\Q$ with $\N$ through some effective enumeration, we can assume algorithms can operate on $\Q$.

\begin{definition}A real number $x$ is called \defin{computable} if there is a computable function $f:\NN \to \QQ$ such that \newline $|f(n)-x|<2^{-n};$
%\begin{itemize} \defin{computable} if there is a computable function $f:\NN \to \QQ$ such that \newline $|f(n)-x|<2^{-n};$
%\item \defin{lower-computable} if there is a computable function $f:\NN \to \QQ$ such that
%$f(n)\nearrow x;$
%\item \defin{upper-computable} if there is a computable function $f:\NN \to \QQ$ such that
%$f(n)\searrow x.$
%\end{itemize}
\end{definition}

	Algebraic numbers or  the familiar constants such as $\pi$, $e$, or the Feigembaum constnt  \cite{Hoy09}  are all computable. However, the set of all computable numbers $\RR_C$ is necessarily
countable, as there are only countably many computable functions. %Lower  (or upper)-computable numbers are also called \emph{left (or right)}-computable. 
%It is straightforward to see that a number is computable if it is simultaneously left- and right-computable. 

%It is easy to present an example of a non-computable left- or right-computable real number. For instance, define the Halting predicate
%$p(i)$ to be equal to $1$ if $A_i$ halts and $0$ otherwise. The number
%$$\alpha=\sum_{n=1}^\infty 10^{-n}p(n)$$
%is evidently not computable. To see that it is left computable, let $p(i,j)$ be the predicate expressing the truth of the sentence
%``$A_i$ halts on step $j$'', and set
%$$\alpha_k=\sum_{n=1}^k\sum_{j=1}^k 10^{-n}p(n,j).$$
%Naturally, $\alpha_k\nearrow \alpha.$

For more general objects, computability is typically defined according to the following principle: object $x$ is computable if there exists an algorithm $\A$ which, upon input $n$, outputs a finite suitable description of $x$  at precision  $2^{-n}$. In this case we say that algorithm $\A$ \emph{computes} object $x$. 

For instance, computability of compact subsets of $\RR^\ell$ is defined as follows. Recall that {\it Hausdorff distance} between two
compact sets $K_1$, $K_2$ is
$$\dist_H(K_1,K_2)=\inf_\eps\{K_1\subset U_\eps(K_2)\text{ and }K_2\subset U_\eps(K_1)\},$$
where $U_\eps(K)=\bigcup_{z\in K}B(z,\eps)$ stands for an $\eps$-neighbourhood of a set.

We say that $K\Subset\RR^\ell$ is {\it computable} if there exists an algorithm $A$ which, upon input $n\in\NN$, outputs a
finite set $C_n \subset \Q^\ell$
of points with rational coordinates such that
$$\dist_H(C_n,K)<2^{-n}.$$

An equivalent, and more intuitive, way of defining a computable set is the following. Let us say that a {\it pixel} is a dyadic
cube with side $2^{-n}$ and dyadic rational vertices. A set $K$ is computable if there exists an algorithm $\A$ which given a pixel
with side $2^{-n}$ outputs $0$ if the center of the pixel is at least $2\cdot 2^{-n}$-far from $K$, outputs $1$ if the center is at most
$2^{-n}$-far from $K$, and outputs either $0$ or $1$ in the ``borderline'' case. In other words, we can visualise $K$ on a computer screen and zoom-in with arbitrarily high magnification.

In this paper
we will speak of {\it uniform computability} whenever a group of computable objects (functions, sets, etc) is computed by a single
algorithm:

\medskip
\noindent
\emph{the objects $\{x_\gamma\}_{\gamma\in\Gamma}$ are computable \defin{uniformly on
a countable set $\Gamma$} if there exists an algorithm $\cA$ with an input $\gamma\in\Gamma$, such that for all
$\gamma\in \Gamma$, $\cA_\gamma:=\cA(\gamma,\cdot)$ computes $x_{\gamma}$}.\\

For instance, a sequence $x_{n}$ of computable points is \emph{uniformly computable} if there is a single algorithm $\A$ which for every $n$ and $m$ outputs a rational number satisfying $|\A(n,m)-x_{n}|<2^{-m}$. 

%To define a computable real-valued function we need to introduce another notion. We say that a function $\phi:\NN\to\QQ$ is an {\it oracle} for $x\in \RR$ if for every $m\in \NN$
%$$d(\phi(m),x)<2^{-m}$$

%On each step, an algorithm may {\it query} an oracle by reading the value of the function $\phi$ for an arbitrary $m\in\NN$. 

%Let $S\subset \RR$. Then a function $f:S\to \RR$ is called {\it computable} if
%there exists an algorithm $\cA$ which, when given access to any oracle for $x\in S$, computes $f(x)$ in the sense that on input $n\in\NN$, after possibly querying the oracle finitely many times, outputs a rational number $s_n$
%such that $|s_n-f(x)|<2^{-n}.$

Open sets can be described by means of {\it rational balls}: balls with rational centres and radii. An open set $A\subset \RR$ is called {\it lower-computable} or \emph{recursively enumerable (r.e.)} if it is the union of a computable sequence of rational balls.  A function $f$ is a \emph{computable function} on some set $S\subset \RR$ if the preimages of rational balls are uniformly lower-computable open (in $S$) sets. That is, if there are uniformly lower-computable open sets $U_n$ such that $f^{-1}(B_n) = U_n \cap S$, where $(B_n)_n$ is an enumeration of all the rational balls.  It can be verified that this definition of computability for a function $f$ is equivalent to being able to compute $f$ in the following sense:  given an arbitrarily good approximation of the input of $f$ in $S$, it is possible to algorithmically
approximate the value of $f$  with any desired precision.  Computability of functions and open sets of $\CC$, $\RR^n$, etc\dots , is defined in
a similar fashion. We refer to \cite{Wei}.

We will use the following terminology. A compact set $K$ is \emph{lower computable} if there is a sequence $x_{n}\in K$ of uniformly computable points which is dense in $K$. It is called \emph{upper-computable} if  its complement is a lower-computable open set.

\begin{example}\label{upper_comp_filled_julia} The filled Julia set $K(P)$ of a computable polynomial  $P$ on $\C$ is always upper computable. For, let $B$ be  a closed rational ball containing $K$. Then, $\C\setminus K=\bigcup_{n\in \N} P^{-n}(\C\setminus B)$ which, since $\C\setminus B$ is a recursively enumerable open set and $P$ is computable, is an upper computable set.   
\end{example}

\begin{example}\label{lower_comp_julia}The Julia set $J(P)$ is always a lower computable set. Indeed, it is not hard to see that the set of repelling periodic points of a computable polynomial $P$ can be algorithmically enumerated (periodic points are uniformly computable, as well as their multipliers) and it is well known that this is a dense subset of $J(P)$.

\end{example}

The following well known characterization of computable compact sets will be used in the sequel. 

\begin{proposition}\label{computableset}A compact set $K$ is computable if and only if it is simultaneously  lower and upper computable. \end{proposition}

As an immediate corollary, we obtain computability of some Julia sets (see \cite{BBY1}): 

\begin{corollary}Let $P$ be a computable complex polynomial such that the filled Julia set $K(P)$ has empty interior. Then, the Julia set $J(P)=K(P)$ is computable. 
\end{corollary}

However, in general, Julia sets need not be computable sets, as it was shown in \cite{BY}:

\begin{theorem}\label{non_comp_julia}There exists  computable parameters $\lambda$, with $|\lambda|=1$, such that the Julia set of the polynomial $\lambda z + z^2$ is not computable.  
\end{theorem}

This result will play an essential role in the proof of our Main Theorem. 

\subsection{Some lemmas on computable maps}
Here we gather a number of elementary results in computable analysis that will be required later in the paper.  

\begin{lemma}\label{l:comp_compact} Let $K\in \C$ be a compact set. Suppose $K$ is computable. Then there exists an algorithm which takes as input any finite list of rational balls $\{B_{n_1},\dots,\B_{n_k}\}$ and halts if and only if they cover $K$.  In this case, we say that the relation $K \subset \bigcup_{i=1}^k B_{n_i}$ is \emph{semi-decidible}. 
\end{lemma}
\noindent
The proof is straightforward and will be left to the reader.

\begin{lemma}[Computable extension]\label{l:modulus}Suppose $K\subset \C$ is a lower computable compact set. Let $\phi : K \to \C$ be a continuous function which is computable on a dense collection of points in $K$ which are uniformly computable. Suppose in addition that $\phi$ has a computable modulus of continuity, that is, there is a computable function $m:(0,a) \to (0,a)$ which is non deceasing and satisfies 
$$
\lim_{\delta \to 0}m(\delta)=0 \quad \text{ and } \quad  |\phi(x)-\phi(y)| < m(|x-y|)
$$
for all $x,y$ in $K$.  Then, $\phi$ is computable. 
\end{lemma}
\begin{proof} Let $\{z_n\}_{n\in \N}$ be the dense set on which $\phi$ is computable and let $z$ be any point in $K$. To compute $\phi(z)$ at a given precision, it suffices to compute $n$ such that $m(|z_n - z|)$ is small enough, and then output $\phi(z_n)$ at a sufficiently high precision.  
\end{proof}

\begin{lemma}[Computable inverse]\label{l:comp_inverse} Let $\Omega\subset \C$ be a lower computable domain and let $K$ be a computable compact set in $\Omega$. Let $f:\Omega \to f(\Omega)$ be a homeomorphism which is computable on $K$.  Then, the inverse $f^{-1}: f(K) \to K$ is also a computable homeomorphism.  
\end{lemma}
\proof
Let $y\in f(K)$ be a given point. We show how to compute $x$ from $y$ such that $f(x)=y$. The set $\Omega\setminus \{y\}$ is recursively enumerable, uniformly in $y$. Since $f$ is computable on $K$, there is a recursively open set $U_y$ such that $f^{-1}(\Omega\setminus \{y\}) = U_y \cap K$. But since $f$ is a homeomorphism,   $U_y \cap K = K \setminus \{x\}$ for some $x$. Note that now we can semi-decide whether $x \in B$ for any rational ball $B$, since this is the case if and only if $B$ together with $U_y$ form a covering of $K$. To compute $x$ at a given precision, just enumerate all balls with diameter less than this precision and semi-decide whether they contain $x$. 
\endproof

\begin{lemma}[Computable images]\label{l:non_comp_image}
Let $\Omega\subset \CC$ be a lower computable domain, and let $K\subset \Omega$ be a computable compact set. Let $f:\Omega \to f(\Omega)$ be a continuous function. Then, if $f$ is computable on $K$, its image $K'= f(K)$ is a computable compact set. 
\end{lemma}
\begin{proof}
Since $K$ is in particular lower computable, we can uniformly compute a sequence of points $x_n \in K$ which is dense in $K$. The sequence $f(x_n) \in K'$ is therefore a computable sequence which is dense in $K'$. This shows that $K'$ is lower-computable. Since $K$ is upper-computable, its complement $\Omega\setminus K$ is a r.e. open set. We now show how one can enumerate a sequence of rational balls in $\CC$ whose union exhausts $\CC\setminus K'$, thus proving computability of $K'$. Let $B$ be a rational ball in $\CC$ and denote by $\cl B$ its closure. It is easy to see that $\cl B$ is disjoint from $K'$ if and only if $K'\subset \C\setminus \cl B$. Note that $\C\setminus \cl B$ is a r.e. open set.  Now,  computability of $f$  on $K$ means that for any r.e. open $U'\subset \CC$ one can uniformly lower compute a set $U\subset \CC$ satisfying $f^{-1}(U' \cap K') = U \cap K$.  In particular, this implies that a r.e. open set $U' \subset \CC$ covers $K'$ if and only if $U$ covers $K$, which is a semi-decidable relation when $K$ is computable.  It follows that we can semi-decide if a given r.e. open set $U'$ covers $K'$. This implies that we can enumerate all the balls $B$ whose closure is disjoint from $K'$, which constitutes a list of balls exhausting the complement of $K'$, and the lemma is proved. 
\end{proof}

%%Examples of computable functions f:\RR\to\RR$ include most common functions such as polynomials (with computable coefficients), trigonometric functions or the exponential.  It is easy to see that a function $f$ is computable iff the pre-images $f^{-1}(U_{n})$ of a sequence $(U_{n})_{n}$ of open uniformly lower-computable sets are again open lower-computable, uniformly in $n$. It follows that computable functions are continuous.  In the following, $X$ will denote a complete computable metric space.
%%%%%%%%%%%%%%%%%%%%%%%%%%%%%%%%%%%%%%%%%%%%%%%%%%%%%%%%%%%%%%%%%

\medskip
\noindent

% !TEX root = Mandelbrot3D.tex

\section{Preliminaries on dynamics of complex polynomials.}
In this section we introduce the tools of complex dynamics that will be used in the proof of our main result.  We refer the reader to \cite{Mil} for an in-depth introduction into the subject; the specific facts on the dynamics of $f_b$ can be found in  \cite{BH}. 

\subsection{Green's function and B{\"o}ttcher coordinate.}

Let $d$ be a positive integer larger than 1, and let~$f$ be a complex  monic polynomial of degree $d$. Denote by 
by~$K(f)$ the \emph{filled Julia set} of $f$; that is, the set of all points~$z$ in~$\C$ whose forward orbit under~$f$ is bounded in~$\C$.
The set~$K(f)$ is compact and its complement is the connected set consisting of all points whose orbit converges to infinity in the Riemann sphere.
Furthermore, we have $f^{-1}(K(f)) = K(f)$ and~$f(K(f)) = K(f)$. 
The boundary~$J(f)$ of~$K(f)$ is the \emph{Julia set of~$f$}.

Recall that the \emph{Green's function of~$K(f)$} is the function $G_f:\C \to [0,+\infty)$ that is identically~$0$ on~$K(f)$ and that for~$z$ outside~$K(f)$ is given by the limit
\begin{equation}
\label{def:Green function}
  G_f(z) = \lim_{n\rightarrow +\infty} \frac{1}{d^n} \log |f^n(z)| > 0.
\end{equation}
The function~$G_f$ is continuous, subharmonic, satisfies~$G_f \circ f= d \cdot G_f$ on~$\C$, and it is harmonic and strictly positive outside~$K(f)$.

It is easy to see that the Julia set of a complex polynomial is connected if and only if every critical point has a bounded orbit.
In this case, the unique conformal isomorphism
$$\varphi_f:\riem\setminus \overline{\DD}\longrightarrow \riem\setminus K(f)\text{ with }\varphi_f(\infty)=\infty,\text{ and }\varphi_f'(\infty)=1$$
conjugates $f$ to $z\mapsto z^d$.
It is called \emph{the (normalized) B{\"o}ttcher coordinate of~$f$ at infinity}  and satisfies $G_f = \log |\varphi_f|$.

The definition of the B{\"o}ttcher coordinate can be extended to the case of a disconnected Julia set as follows. It is well known that  $K(f)$ is connected if and only if all critical values of $f$ lie inside $K(f)$. Let $\omega$ be the critical value of $f$ such that $G_f(\omega)$ is maximal. 
Then the domain
$$ U_f := \{z\in \C \mid G_{f}(z) > G_{f}(\omega)\} $$
is homeomorphic to a punctured disk. 
We then define $\varphi_f$ as the unique conformal isomorphism
\[
 \varphi_{f}: U_{f}
\rightarrow
\{z\in \riem \mid |z| > \exp (G_{f}(\omega)) \},
\]
with $\varphi_{f}(\infty)=\infty$ and $\varphi_{f}'(\infty)=1$. It is not hard to see that $\varphi_{f}$ still conjugates~$f$ to $z \mapsto z^d$.

Let $S_f$ be the union of the critical points of $G_f$  in $\mathbb{C}\setminus K_f$ and the stable manifolds of the gradient flow of $G_f$ on $\mathbb{C}\setminus K_f$.
Denote  $V_f$ the open set $\mathbb{C}\setminus (K_f\cup S_f)$.
The  B\"ottcher coordinate  $\varphi_f$ extends to an analytic map $\varphi_f: V_f\to \mathbb{C}$ and satisfies $G_f(z) = \log |\varphi_f(z)|$ on $V_f$.

By definition, for~$v > 0$ the \emph{equipotential of level~$v$ of~$f$} is the set $G_f^{-1}(v)$.
A \emph{Green's line of~$G_f$} is a smooth curve on the complement of~$K(f)$ in~$\C$ that is orthogonal to the equipotentials of~$G_f$ and that is maximal with this property. Note that in the case when $K(f)$ is connected, every Green's line must accumulate inside the Julia set $J(f)$. If $K(f)$ is not connected, some Green's lines will terminate at escaping critical points of $f$ and their preimages.

Given~$t$ in~$\R / \Z$, the \emph{external ray of angle~$t$ of~$f$}, denoted by~$R_f(t)$, is the Green's line of~$G_f$ containing
$$ \{ \varphi_f^{-1}(r \exp(2 \pi i t)) \mid \exp(G_f(0))< r < +\infty \}. $$
By the identity~$G_f \circ f_c= d\cdot G_c$, for each~$v > 0$ and each~$t$ in~$\R / \Z$ the map~$f$ maps the equipotential~$v$ to the equipotential~$d\cdot v$ and maps~$R_f(t)$ to~$R_f(d\cdot t)$.
For~$t$ in~$\R / \Z$ the external ray~$R_f(t)$ \emph{lands at a point~$z$}, if~$G_f : R_f(t) \to (0, + \infty)$ is a bijection and if~$G_f|_{R_f(t)}^{-1}(v)$ converges to~$z$ as~$v$ converges to~$0$ in~$(0, + \infty)$.
By the continuity of~$G_f$, every landing point is in $J(f) = \partial K(f)$.

\noindent
We use the following simple fact several times.
\begin{lemma}\label{l:landing}
Let~$f$ be a complex monic polinomial of degree $d\ge 2$, let~$t$ be in~$\R / \Z$ and suppose that the external ray~$R_f(t)$ lands at a point~$z_0$ of~$K(f)$ which is not a critical value of $f$; so~$f^{-1}(z_0)$ consists of~$d$ distinct points.
Then each point of~$f^{-1}(z_0)$ is the landing point of precisely one of  the external rays~$R_f((t + k)/d)$, for $k\in \{1, 2, \ldots, d-1\}$.
\end{lemma}

\subsection{Dynamics of maps $z\mapsto \lambda z+bz^2+z^3$}

Let us fix a $\lambda\in S^1$, and consider the family $$f_b(z)=z+bz^2+z^3$$ as above.
For every $b$ in $\C$ the polynomial $f_b$ has two critical points (counted with multiplicity) and one indifferent fix point at 0.  It is known that the presence of this indifferent fixed point forces at least one of these critical points to have a bounded orbit. When $b^2=3\lambda$, these two critical points are equal, and therefore both have a bounded orbit. It follows that $b\in \M_\lambda$.  When $b^2\neq 3\lambda$, the two critical points are different.

In the case when $b\notin \M_\lambda$, let us denote $\omega_1$ the critical point with bounded orbit, and $\omega_2$ the other, escaping, critical point.
The critical value $f_b(\omega_2)$ has two preimages. We call  \emph{co-critical} point the preimage of $f_b(\omega_2)$ which is different from $\omega_2$, and we denote it by $\omega'_2$.

The map
\[
 \begin{array}{cccl}
  \Phi_\lambda : &\C \setminus \M_\lambda & \to & \C \setminus \overline{\D}\\
          &     b         &  \mapsto           & \Phi_\lambda(b)=\varphi_b(\omega'_2).
 \end{array}
\]
is a conformal isomorphism.
For~$v > 0$ the \emph{equipotential~$v$ of~$\M_\lambda$} is by definition
$$ \cE_\lambda(v) := \Phi_{\lambda}^{-1}(\{z\in \C \mid |z| = v \}). $$ 
On the other hand, for~$t$ in~$\R / \Z$ the set
$$ \cR_\lambda(t) := \Phi_\lambda^{-1}(\{r \exp(2 \pi i t) \mid r > 1 \}). $$
is called the \emph{external ray of angle~$t$ of~$\M_\lambda$}.
We say that $\cR_{\lambda}(t)$ \emph{lands at a point~$b$} in~$\C$ if~$\Phi^{-1} (r \exp(2\pi i t))$ converges to~$z$ as $r \searrow 1$.
When this happens~$z$ belongs to~$\partial \M_\lambda$.

 Let $b_1$ be the parameter with potential $\eta=1/3$ and external angle $\theta = 1/4$.  Let $U$ denote the open set $\{z\in \C \mid G_{b_{1}}(z)< 3G_{b_{1}}(\omega_2)\}$. Note that the equipotential of level $3G_b(\omega_2)$ is a real-analytic simple closed curve, and thus $U$ is a topological disk. The set $f_{b_1}^{-1}(U)$ is the set  $\{z\in \C          \mid         G_{b_{1}}(z)< G_{b_{1}}(\omega_2)\}$ which is bounded by a lemniscate pinching at the escaping critical point $\omega_2$. Let $U'$ be the connected component of $f_{b_1}^{-1}(U)$ that contains the non-escaping critical point $\omega_1$. We will denote by $Q_{b_1}:U' \to U$ the restriction of $f_{b_1}$ to $U'$. The filled Julia set $K(Q_{b_1})$ is defined as the set of points in $U'$ that remain in  $U'$ under iterations by $Q_{b_1}$. The Julia set $J(Q_{b_1})$ is the boundary of $K(Q_{b_1})$.  

 \noindent
 The following result, extracted from \cite{BH}, states that $J(Q_{b_1})$ is a quasi-conformal copy of $J(\lambda z + z^2)$.

 \begin{thm}\label{quasi_conformal} There exist a quasi-conformal homeomorphism $\phi: \C \to \C$ which conjugates $Q_{b_1}$ to $\lambda z + z^2$ on their Julia sets.  
 \end{thm}

 \noindent
 It will not be necessary to give the definition of a quasi-conformal homeomorphism here since all what we will need is the following  standard property of such maps (see e.g. \cite{Ahlfors-book}):
 
 \begin{proposition}\label{l:holder} A quasi-conformal map $\phi$ from a topological disk $D$ into itself is H{\"o}lder-countinuous. More precisely, there exist constants $H,\alpha$ such that for every $x,y$ in $D$
 $$
 |\phi(x) - \phi(y)| \leq H |x-y|^\alpha .
 $$
%% Moreover, the constant $H$ can be chosen to depend only on $D$. 
 \end{proposition}

Buff and Henriksen also give a characterisation of $J(Q_{b_1})$ as the landing points of a particular set of dynamical rays that we now describe.  
%First of all, they show that if $\lambda = \neq 1$, then the two dynamical rays $R_{b_1}(0/1)$ and $R_{b_1}(1/2)$ both land at a common fixed point $\beta \neq 0$ which is repelling.  Therefore, the curve $\{\beta\} \cup R_{b_1}(0/1)  \cup R_{b_1}(1/2)$ cuts the plane into two connected components $V_1$ and $V_2$. 
Let $\Theta\subset \RR/\ZZ$ be the set of angles $\theta$ such that for every integer $n\ge 0$ we have $3^n\theta \in [0,1/2]$ mod 1. It is a Cantor set forward invariant under multiplication by $3$. It is shown in \cite{BH} that for any $\theta \in \Theta$, the dynamical ray $R_{f_{b_1}}(\theta)$ does not bifurcate, and that the set defined by
\[
X_{b_1} = \bigcup_{\theta \in \Theta} R_{f_{b_1}}(\theta)
\]
satisfies $\overline{X_{b_1}}\setminus X_{b_1}  = J(Q_{b_1}) \subset J(f_{b_1})$.

\subsection{Julia sets in $\M_{\lambda}$}

The parameter rays $\cR_\lambda(1/6)$ and $\cR_\lambda(1/3)$ both land at the parameter $b_0=4(\lambda-1)$, see \cite{BH}. The wake $\cW_0$ is defined as the connected component of 
$$
\CC\setminus (\overline{\cR_\lambda(1/6)}\cup \overline{\cR_\lambda(1/3)})
$$
containing the ray $\cR_\lambda(1/4)$. 

%Define $J_b$ to be the set  $J_b = \overline{X_b}\setminus X_b$ and $K_b$ to be the complement of the unbounded connected component of $\CC\setminus J_b$. 
%Then, $K_b$ is contained in the filled-in Julia set $K(f_b)$, its boundary $J_b$ is contained in the Julia set $J(f_b)$ and $K_b$ is quasi-conformally homeomorphic
%to the filled-in Julia set $K(\lambda z + z^2)$.

Every dyadic number $\vartheta=(2p+1)/2^k, k\ge 1$ and $0<2p+1<2^k$ can be expressed in a unique way as a finite sum
$$
\frac{2p+1}{2^k} = \sum_{i=1}^k \frac{\varepsilon_i}{2^i},
$$
where each $\varepsilon_i, i=1,\ldots, k$ take the value 0 or 1. We define  $\vartheta^-$ and $\vartheta^+$ by the formulae:
$$
\vartheta^- = \sum_{i=1}^k \frac{\varepsilon_i+1}{3}, \text{ and } \vartheta^+ = \vartheta^- + \frac{1}{2\cdot 3^k}.
$$

\noindent\textbf{Proposition 12 in \cite{BH}.}\textit{ Given any dyadic angle $\vartheta=(2p+1)/2^k, k\ge 1, 0<2p+1<2^k,$ 
the two parameter rays $\cR_\lambda(\vartheta^-/3)$ and $\cR_\lambda(\vartheta^+/3)$ land at a common point $b_\vartheta$.
 Moreover,
$$
f_{b_\vartheta}^{k+1}(\omega_2(b_{\vartheta}))=\beta(b_\vartheta).
$$
}
%and the two dynamical rays $R_{b_\vartheta}(\vartheta^-/3)$ and $R_{b_\vartheta}(\vartheta^+/3)$ land at $\omega'_2(b_\vartheta)$.}

\medskip

The wake $\cW_\vartheta$ is defined as the connected component of 
$$
\CC\setminus (\overline{\cR_\lambda(\vartheta^-/3)}\cup \overline{\cR_\lambda(\vartheta^+/3)})
$$
that contains the parameter ray $\cR_\lambda(\theta)$ with $\theta$ in $]\vartheta^-/3, \vartheta^+/3[$.  We now can define $\mathcal{X}_\vartheta$ to be the set of parameter rays
$$
\mathcal{X}_\vartheta = \bigcup_{\theta\in \Theta} \cR_\lambda\left(\frac{\vartheta^-}{3} + \frac{\theta}{3^{k+1}} \right),
$$
and let $\mathcal{J}_\vartheta$ to be the set $\mathcal{J}_\vartheta = \overline{\mathcal{X}_\vartheta}\setminus \mathcal{X}_\vartheta$, where the closure is taken in $\CC$.

%Define $\mathcal{K}_\vartheta$ to be the complement of the unbounded connected component of $\CC\setminus \mathcal{J}_\vartheta$.

Let $h:\cW_0\times \CC \to \CC$  be a quasi-conformal extention  of the holomorphic motion 
$h:\cW_0 \times X_{b_1}\to \CC$ defined by $h_b(z)=\varphi_b^{-1}\circ \varphi_{b_1} (z)$. By \cite[Lemma 13]{BH} the map $H_\vartheta: \cW_0 \to \CC$ defined by
$$
H_\vartheta(b) = h_b^{-1}(f_b^{k+1}(\omega_2(b)))
$$
is locally quasi-regular, and its restriction to the dyadic wake $\cW_\vartheta$ is a locally quasi-conformal homeomorphism sending $\mathcal{J}_\vartheta$ to $J(Q_{b_1})$.

 % !TEX root = Mandelbrot3D.tex

\section{Proof of the Main Theorem}

\subsection{Computable B\"ottcher's coordinate}

Let 
$$
f(z) = z^d + a_{d-1}z^{d-1} + \cdots +a_1 z + a_0
$$ 
be a polynomial of degree $d\ge 2$. Let $\varphi_f$ be the B\"ottcher's coordinate of $f$ at infinity, and, as before, let
$$V_f=\riem\setminus (K_f\cup S_f),$$
where $S_f$ is the union of the critical points of $G_f$  in $\mathbb{C}\setminus K_f$ and the stable manifolds of the gradient flow of $G_f$ on $\mathbb{C}\setminus K_f$.
The main result of this subsection is the following.

\begin{prop}[Computability of B\"ottcher's coordinate]
\label{p: computability Bottcher}
Let $f$ be a computable  monic polynomial  of degree $d\ge 2$. The open set $V_f$ is lower-computable and the B\"ottcher coordinate $\varphi_f$ is computable on $V_f$. 
\end{prop} 
The proof of this proposition will be given after the following sequence of lemmas.

\begin{lemma}
%[Computability of the local Bottcher's coordinate at infinity]
There is $R>0$ such that the   B\"ottcher's coordinate is computable on $\mathbb{C}\setminus \overline{\mathbb{D}}_R$. 
\end{lemma}
\proof For $|z|$ sufficiently large the B\"ottcher coordinate can be written as a infinite product as follows:
$$
\varphi_f(z) = z \cdot \prod_{n=0}^{+\infty} \left(1+ \frac{f(f^{n}(z))-(f^n(z))^d}{(f^n(z))^d}\right)^{d^{-(n+1)}}.
$$
For example, if $R\ge \max\{4 \sum_{j=1}^{d-1} |a_j|, 4/3\}$, then by induction we have that for $|z|\ge R$,
$$
\frac{|f(f^{n}(z))-(f^n(z))^d|}{|f^n(z)|^d} \le \frac{1}{4},
$$
and thus, the principal value of the $d^{n+1}$-root is defined. Taking logarithm  of the absolute value one can see that the corresponding series converges and thus,
the product also converges.  
For computing the rate of convergence put 
$$a_n=\frac{f(f^{n}(z))-(f^n(z))^d}{(f^n(z))^d}.$$
Notice that 
$$
\log\frac{3}{4} \le \log |1+a_n| \le \log\frac{5}{4}.
$$
This implies that 
$$
  \left| \prod_{n=0}^{k-1}\left(1+a_n\right)^{d^{-(n+1)}}\right | \le \frac{5}{4}^{\frac{1}{d-1}},
$$
and
$$
\frac{3}{4}^{\frac{1}{d^k(d-1)}} \le  \left| \prod_{n=k}^{+\infty}\left(1+a_n\right)^{d^{-(n+1)}}\right | \le \frac{5}{4}^{\frac{1}{d^k(d-1)}}.
$$
Thus,
\begin{multline*}
  \left|\varphi_f(z)  -  z\cdot \prod_{n=0}^{k-1}\left(1+a_n\right)^{d^{-(n+1)}}\right | \\ 
   \le 
  \left| z\cdot \prod_{n=0}^{k-1}\left(1+a_n\right)^{d^{-(n+1)}}\right | \left| \prod_{n=k}^{+\infty}\left(1+a_n\right)^{d^{-(n+1)}}-1 \right | 
  \le 
  |z|\frac{5}{4}^{\frac{1}{d-1}}\left|\frac{5}{4}^{\frac{1}{d^k(d-1)}}-1\right|.
\end{multline*}

\endproof
\begin{lemma}
%[Computability of the Green's function]
The Green's function $G_f$ is computable on $\mathbb{C}\setminus K_f$.
\end{lemma}
\proof
For $z$ in $K_f$ and for every $k$ in $\mathbb{N}$ we have  
$$
G_f(z) =  \frac{1}{d^{k}}  \log|f^k(z)| + \sum_{n=k+1}^{+\infty} \log\left(\frac{|f^{n+1}(z)|}{|f^n(z)|^d}\right)^{d^{-(n+1)}} 
$$
For $R\ge \max\{4 \sum_{j=1}^{d-1} |a_j|, 4/3\}$ we have by induction that if  $|z|\ge R$ 
then
$$
\frac{3}{4}\le \frac{|f^{n+1}(z)|}{|f^n(z)|^d} \le \frac{5}{4}.
$$
This implies that 
$$
 \frac{1}{d^k(d-1)} \log\frac{3}{4} \le\left| G_f(z) -  \frac{1}{d^{k}}  \log|f^k(z)| \right| \le  \frac{1}{d^k(d-1)} \log\frac{5}{4}.
$$
\endproof

\proof[Proof of Proposition \ref{p: computability Bottcher}]

Let $R$ be a positive number such that the B\"ottcher coordinate is computable on $\mathbb{C}\setminus \overline{\mathbb{D}}_R$. 
Now consider the flow $(z,t)\to F(z,t)$ associated to the gradient vector field 
$\nabla G_f$ on the complement of $K_f$. 
Since $\nabla G_f$ is analytic the dependence on $z$ of the flow $F(z,t)$  is also analytic. 
Observe that for every $t>0$ we have that 
$$
G_f(F(z,t)) = G_f(z) + \int_{0}^{t} | \nabla G_f(F(z,s)) | ^2 ds.
$$
Thus, for every $z$ in $V_f$ there is $t\ge 0$ sufficiently large such that  $|F(z,t)|> R$. 
It follows that the map 
$$
z\in V_f\to  \exp(G_f(z)) \frac{\varphi_f (F(z,t))}{|Ê\varphi_f (F(z,t)) |}. 
$$  
is a holomorphic extension of the B\"ottcher coordinate to $V_f$ and so it must be equal to $\varphi_f$  on $V_f$.
On the other hand, since $G_f$ is computable and analytic it follows that $\nabla G_f$ is also computable and effectively locally Lipschitz 
on the complement of $K_f$ (see  \cite[Theorem 2]{pour2017computability} and \cite[Theorem 1]{gracca2009computability})  which is recursively enumerable open. 
Thus,  by \cite[Theorem 3]{gracca2009computability} for every $z\notin K_f$ the map $t\in [0,+\infty) \to F(z,t)$ is computable. This implies that we can semi-decide whether $|F(z,t)|> R$, which is equivalent to say that $V_f$
is lower-computable open. Moreover, using that $G_f$ is computable we conclude that the extension of the B\"ottcher coordinate on $V_f$ is also computable.
\endproof

\subsection{Computable external rays and their landing points}
%Let us prepare the proof with a number of intermediate results. 

\begin{lemma}[Computable inverse branches]\label{l:branches} Let $f$ be a computable polynomial of degree $d$ and let $\beta$ be a fixed point of $f$ which is not a critical value. Then, one can uniformly compute positive real numbers $r_0, r_1, \dots, r_{d-1}$ and points $\beta=\beta_0, \beta_1, \dots, \beta_{d-1}$ in $\C$ such that:
\begin{itemize}
\item $f(\beta_i)=\beta$ for $i=0,\dots,d-1$, 
\item the open disks $D(\beta_i,r_i)$, $i=0,\dots,d-1$ are pairwise disjoint and
\item $f$ restricted to each $D(\beta_i,r_i)$ is conformal and $f(D(\beta_i,r_i))\subset D(\beta,r_0)$ for $i=0,\dots,d-1$. 
\end{itemize}
Moreover, the  inverse branches $g_i:  f(D(\beta_i, r_i)) \subset D(\beta, r_0) \to D(\beta_i,r_i)$  of $f$ are all computable, uniformly in $i$. 
\end{lemma}

\begin{proof} 
Since $\beta$ is not a critical value, it has exactly $d$ different preimages, one of which is $\beta$ (since it is fixed). Let $\beta_1,...,\beta_{d-1}$ be the other preimages.  Since  $f$ has finitely many critical values, all of which are computable, we can compute $r_0$ such that $D(\beta,r_0)$ is at some positive distance away from the collection of critical values. Then, the open set $f^{-1}(D(\beta,r_0))$ consist of exactly $d$ connected components, each of which contains one of the $\beta_i$, $i=0,\dots,d-1$. Clearly, now we can compute numbers $r_i$ such that $D_i=D(\beta_i,r_i)$ is included in the component containing $\beta_i$.  By construction, $f$ is conformal on each $D_i$. Moreover, by Theorem 4.5 from \cite{Hertling}, $f|_{D_i}:D_i \to f(D_i)\subset D(\beta,r_0)$ and its inverse $g_i: f(D_i)\to D_i$ is also computable, uniformly in $i$. The lemma is proved. 
\end{proof}

Let $\theta \in \R / \Z$ be such that $R_f(\theta)$ lands.  For an interval $I\subset (1,\infty)$ we will denote by $R_{f}^{\theta}(I)$ the \emph{ray segment} defined by
$$
R_{f}^{\theta}(I) = \varphi_f^{-1}\{ re^{2\pi i \theta} : r \in I   \}.
$$

\begin{lemma}[Effective landing]\label{l:effective_landing}   Let $f$ be a computable polynomial of degree $d$ and let $\beta$ be a fixed point of $f$ such that $|df(\beta)|>1$.  Suppose that the dynamical ray $R_f(\theta)$ lands at $\beta$ and that $\theta$ is a computable angle. Then the set $R^{\theta}_{f}(1,2] \cup \{\beta(b)\} $  is a computable compact set. 
\end{lemma}
\begin{proof} By Proposition \ref{p: computability Bottcher} we can compute a dense sequence of points in $R_{f}^{\theta}(1,2]$, for instance by computing the sequence $\varphi^{-1}_{f}(q_n)$ where $(q_n)$ is a computable sequence of rationals which is dense in $(1,2]$. This sequence is of course also dense in $R^{\theta}_{f}(1,2] \cup \{\beta\} $, which is therefore a lower-computable set. We now show that it is also upper computable.  Since $|df(\beta)|>1$, it follows that $\beta$ is an attracting fixed point for the the inverse branch $g=f^{-1}$ of $f$ that leaves $\beta$ fixed. Moreover, we can compute a neighbourhood $T_{\beta}$ of $\beta$ such that its closure $\overline{T_{\beta}}$ shrinks to $\{\beta\}$ under iterates by $g$. Indeed, we could take for instance $T_{\beta}$ to be the open disk centred at $\beta$ with radius $\ln(|df(\beta)|^{|df(\beta)|})$. Since $R_{f}^{\theta}(1,2]$ is lower-computable, we can find a point $z\in R_{f}^{\theta}(1,2]$ which belongs to $T_{\beta}$.  Let $\eta(z) \in (1,2]$ be the level of the equipotential line containing $z$. That is, $\varphi_{f}^{-1}(\eta(z)e^{2\pi i \theta}) = z$. Since the ray $R_{f}(\theta)$ is invariant under iterations by $g$, we have that $g^i(z) \in R_{f}^{\theta}(1,2]$, and since $\varphi_{f}$ conjugates $f$ to $z^d$, we  obtain that 
$$
\eta(g^i(z)) = ( \eta (z) )^{1/d^i}.
$$

Now, compute the ray segment $R_{f}^{\theta}[\eta(g (z)), \eta(z)]$ which goes from $g(z)$ to $z$ in $ R_{f}^{\theta}(1,2]$.  This is clearly a computable closed set. Thus, by lemma \ref{l:comp_compact}, we can semi-decide if it is contained in $T_{\beta}$. In a dovetail fashion, semi-decide whether  $R_{f}^{\theta}[\eta(g^{i+1}z),\eta(g^{i}z)]$ is contained in $T_{\beta}$, for larger and larger $i$. Since $R_{f}(\theta)$ lands at $\beta$, this procedure must eventually stop for some $i^*$. Let $z^*=g^{i^*}(z)$ and denote by $I^*$ the ray segment $R_{f}^{\theta}[\eta(g(z^*)), \eta(z^*)]$. Let $d^*_0 = \max \{ d(z,\beta): z \in I^*\}$ so that $I^*$ is completely contained in $\overline{D}(\beta,d^*_0) \subset T_{\beta}$.   Since the region $T_{\beta}$ is trapping, it follows that all the iterates of $I^*$ by $g$ are contained in $\overline{D}(\beta,d^*_0)$ and thus, so is $R_{f}^{\theta}(1,\eta(z^*)]$.  Let $d^*_i = \max \{ d(z,\beta): z \in g^i(I^*)\}$.  Since $d^*_i \to 0$ as $i\to \infty$, one has that 
$$
\C\setminus \left( R^{\theta}_{f}(1,2] \cup \{\beta\} \right)= \bigcup_{i\in\N} \C \setminus \left( R_{f}^{\theta}[\eta(g^iz^*),2] \cup  \overline{D}(\beta, d^*_i) \right).
$$ 
But the sets in the union of the right-hand side are all recursively enumerable open sets, uniformly in $i$. The lemma follows.  
\end{proof}

\begin{lemma}[Effective pairing]\label{l:rays_pairing}
Let $f$, $\theta \in \R/ \Z$ and $\beta$ be as in the previous lemma. Then, one can uniformly compute angles $\theta_1, \dots , \theta_{d-1}$  in $\R / \Z$ and points $\beta_1, \dots, \beta_{d-1}$ in $K(f)$ such that $\beta_i$ is precisely the landing point of $R_f(\theta_i)$, for $i=1,\dots,d-1$. 
\end{lemma}
\begin{proof}
For $i=0, \dots, d-1$, let $\beta_i$, $r_i$ and $g_i$ as in Lemma \ref{l:branches}.  By Lemma \ref{l:landing}, each $\beta_i$ is the landing point of exactly one of the rays $R_f(\frac{\theta + k}{d})$ for $k=1,\dots,d-1$. Since the angles $\frac{\theta+k}{d}$ are all computable, all we need to do is to decide, for each $\beta_i$, which of the rays is the one landing at $\beta_i$. By Lemma \ref{l:effective_landing}, the set  $R_f^{\theta}(1,2] \cup \{\beta\}$ is a computable closed set. It is easy to see that one can compute $t$ such that $R=R_f^{\theta}(1,t] \cup \{\beta\}$ is contained in $D(\beta,r_0)$.  Now, for each $i$, the set $g_i(R)\subset D(\beta_i,r_i)$ is a computable closed set which is contained in the ray landing at $\beta_i$. To compute $\theta_i$, just choose any point $z\in g_i(R)$ different from $\beta_i$ and compute its external angle. This is $\theta_i$.
 \end{proof}

\subsection{Proof of the Main Theorem}

The proof will follow from the following two lemmas.

\begin{lemma}\label{l:effective_straightening}  Let $\phi$ be the quasi-conformal homeomorphism of Lemma \ref{quasi_conformal}, conjugating the $Q_{b_1}$ to  $P_\lambda=\lambda z + z^2$ on their Julia sets. Suppose that  $\lambda \neq 1$ is computable and $|\lambda|=1$. Then $\phi$ is computable on $J(Q_{b_1})$. 
\end{lemma}
\begin{proof} Recall that by Lemma \ref{quasi_conformal} there exist a quasi-conformal function $\phi : \C \to \C$ that conjugates $Q_b$ to $P(\lambda)$. We show how to compute this function on $J(Q_{b_1})$ by computing it on a dense set of points, and then invoking Lemma \ref{l:modulus} with any bound on the uniform modulus of continuity of $\phi$, which exists because of the Holder property of quasi-conformal maps (see Proposition \ref{l:holder}). The dense set will be given by the $\beta(b_1)$ fixed point of $f_{b_1}$, together with all their preimages under $Q_{b_1}$. Since $\phi$ is a conjugacy, this set is sent to the set of pre-images of the $\beta(\lambda)$ fixed point of $P_\lambda$. Since these fixed points are computable, so are the sets of preimages.  Thus, it is enough to show how to algorithmically decide, for a given preimage of $\beta(b)$, which preimage of $\beta(\lambda)$ it goes to. To achieve this, we use the external arguments of the points: on one side we start with $R_0$ (which lands at $\beta(b_1)$), whose preimages are $R_0$ and $R_{1/3}$ (which lands at $\beta_1$, the preimage of $\beta$ different from it). Then $R_{1/9}$ (which lands at one preimage of $\beta_1$) and $R_{1/9 + 1/3}$ (which lands at the other preimage of $\beta_1$) and so on. On the other side these are $R_0$, then $R_0$ and $R_{1/2}$, then $R_{1/4}$ and $R_{1/4 + 1/2}$  and so on.  By respecting the orientation, we can pair the angles on different sides. If, moreover,  we were able to pair preimages of the $\beta$ fixed point with the external ray landing at them, we could then pair the preimages of the $\beta(b)$-fixed point with the corresponding preimages of the $\beta(\lambda)$ fixed point.  But this is precisely given by Lemma \ref{l:rays_pairing}, and so the proof is finished. 
\end{proof}

\begin{lemma}\label{l:H_comp}The map $H_{\vartheta}: \mathcal{W}_0  \to \C $ is computable on the closure of $\mathcal{W}_0 \setminus \M_\lambda$. Moreover, the restriction of $H_{\vartheta}$ to the closure of $\mathcal{W}_\vartheta\setminus \M_\lambda$ has a computable inverse. 
\end{lemma}
\begin{proof}
 By Proposition \ref{p: computability Bottcher}, the mappings $\varphi_b$, $\varphi_{b_1}$ and their inverses are computable on their domains $U_b, U_{b_1}$, uniformly in $b$ and $b_1$. Hence,  holomorphic motion $h = \varphi^{-1}_b \circ \varphi_{b_1}: U_{b_1} \to  U_{b}$ is computable too. Recall that $H_{\vartheta} = h^{-1}(f^{k+1}(\omega_2(b)))$. Since $f^{k+1}(\omega_2(b)) \in U_b$ for $b \notin \M_\lambda$, to prove computability of the map $H_{\vartheta}$  on $\mathcal{W}_0\setminus \M_\lambda$ it is enough to show that the critical point $\omega_2(b)$ is computable from $b$. But the collection of critical points is always computable from $b$, and we can identify the escaping one.  Now, since $H_{\vartheta}$ is quasi-conformal, it has the holder property on some large enough ball and therefore, by Lemma \ref{l:holder}, its computability can be extended  up to the closure of $\mathcal{W}_0\setminus \M_\lambda$.  It follows that the restriction of $H_{\vartheta}$ to the closure of $\mathcal{W}_\vartheta\setminus \M_\lambda$ is computable homeomorphism.  
Computable inverse will follow from Lemma \ref{l:comp_inverse}. However, note that we can not apply it directly to the closure of $\mathcal{W}_0\setminus \M_\lambda$ because it may not be a computable set (it may not be upper-computable). Instead, we first note that since $\mathcal{W}_0\setminus \M_\lambda$ is a recursively enumerable open set, we can produce a sequence of computable compact sets whose union equals $\mathcal{W}_0\setminus \M_\lambda$.  We can then apply  Lemma \ref{l:comp_inverse} to each of these set, which proves that the inverse is computable on $H_{\vartheta}(\mathcal{W}_0\setminus \M_\lambda)$. But we can now apply  Lemma \ref{l:holder} to this inverse, which proves that its computability can be extended to the closure, as was to be shown. 
 \end{proof}

\begin{lemma}\label{l:non_comp_subset}
If $\partial M_{\lambda}$ is upper-computable, then so is $\mathcal{J}_{\vartheta}$.
\end{lemma}
\begin{proof}
Suppose $\partial M_{\lambda}$ is upper computable.  We only need to show that we can enumerate a sequence of balls in $\C$ whose union exhaust the complement of $\mathcal{J}_{\vartheta}$.  This complement is made by the complement of $\partial M_{\lambda}$, the complement of the mini-wake $\mathcal{W}_{\vartheta}$, and the collection of the unbounded components of $\C \setminus \overline{\mathcal{X}_{\vartheta}}$.  The complement of $\partial M_{\lambda}$ is recursively enumerable by hypothesis.

Recall that the dynamical rays $R_{b_1}(0)$ and $R_{b_1}(1/2)$ both land at the $\beta$ fixed point of $f_{b_1}$. Thus, the curve $\{\beta\} \cup R_{b_1}(0) \cup R_{b_1}(1/2)$ cuts the plane into two connected components $V_1$ and $V_2$. Let $V_2$ be the one containing the escaping critical point $\omega_2$.    To see that the complement of the mini-wake $\mathcal{W}_{\vartheta}$ is also recursively enumerable, we use the fact that $H_{\vartheta}$ maps $\mathcal{W}_{\vartheta}$ (respectively $\partial \mathcal{W}_{\vartheta}$) to $V_1$ (respectively $\partial V_1$).  This is shown in the proof of Lemma 13 from \cite{BH}. In particular, $H_{\vartheta}$ maps the parameter rays $\mathcal{R}_\lambda(\vartheta^-), \mathcal{R}_\lambda(\vartheta^+)$ to the dynamical rays $R_{b_1}(0)$ and $R_{b_1}(1/2)$.  Now, let $B\subset \C$ be some computable ball containing $K(f_{b_1})$ and consider the set $B'_1=\partial (B\cap V_1)$.  By Lemma \ref{l:effective_landing}, it is straightforward to see that the curve $B'_1$ is a computable set. Since the inverse of $H_{\vartheta}$ is computable there (by Lemma \ref{l:H_comp}), we see by Lemma \ref{l:non_comp_image} that $H^{-1}_\vartheta(B'_1)$ is also a computable set.  It is now straightforward to see that the complement of the mini-wake $\mathcal{W}_{\vartheta}$ is recursively enumerable. It remains to show that the collection of unbounded components of $\C \setminus \overline{\mathcal{X}_{\vartheta}}$ is uniformly recursively enumerable. Recall that these components correspond to preimages by $H_{\vartheta}$ of the unbounded components of  $\C\setminus \overline{X_{b_1}}$. But these components are precisely given by the preimages of $V_2$ by iterates of $f_{b_1}$.  Note that, by Lemma \ref{l:effective_landing} again, the set $B'_2 = \partial (B \cap V_2)$ is computable, and using Lemma \ref{l:branches}, we see that their  preimages by $f_{b_1}$ are computable too and thus so are the preimages of these by $H_\vartheta$.  Moreover, by taking any computable point in the bounded component of the complement of $B'_2$, we see that the interior of this last collection can be uniformly enumerated, from which it is straightforward to see that the unbounded components of $\C \setminus \overline{\mathcal{X}_{\vartheta}}$ can be uniformly recursively enumerated, as it was to be shown. 

\end{proof}

\bigskip

 % !TEX root = Mandelbrot3D.tex

We are now  ready to finish the proof of our main result.

\begin{theorem}There exists a computable $\lambda$ such that the bifurcation locus $M_{\lambda}$ is not computable. 
\end{theorem}
\begin{proof}
By Theorem \ref{non_comp_julia}, there exist a computable $\lambda$ such that $J(\lambda z + z^2)$ is not computable. By Lemma \ref{l:effective_straightening} and Lemma \ref{l:non_comp_image}, the Julia set $J_{b_1}$ of $Q_{b_1}$ is not computable either. We prove that the preimage of this set by $H_\vartheta$ is not computable. It is enough to show that the map $H_\vartheta$ is computable on this preimage.  From the proof of Lemma \ref{l:non_comp_subset} we see that $H_\vartheta$ is computable on the closure of $\C \setminus \M_\lambda$, which contains $\mathcal{J}_\vartheta$. By Lemma \ref{l:non_comp_image}, $\mathcal{J}_\vartheta$ is not computable, and the Theorem now follows from Lemma \ref{l:non_comp_subset}.
\end{proof}

\bibliographystyle{plain}
\bibliography{biblio}

\end{document}